\documentclass{article}

\usepackage[french,british]{babel}

\usepackage[T1]{fontenc}

\usepackage{amsmath}

\usepackage{amsfonts}
\usepackage{amsbsy}
\usepackage{amssymb}

\usepackage{amsthm}

\theoremstyle{plain}

\usepackage{latexsym}
\usepackage{amsmath}
\usepackage{amsfonts}
\usepackage{amssymb}
\usepackage{psfrag}
\usepackage{graphicx}

\usepackage{amsmath,amssymb}
\usepackage{graphicx}

\usepackage{amsmath,amssymb}
\usepackage{graphicx}

\newtheorem{ozn}{Definition}[section]
\newtheorem{thm}{Theorem}[section]

\newtheorem{nas}{Corollary}[section]
\newtheorem{zau}{Remark}[section]
\newtheorem{lema}{Lemma}[section]
\newtheorem{pry}{Example}[section]

\newcommand{\eps}{\varepsilon}
\newcommand{\me}{\mathbf}
\newcommand{\mr}{\mathbb}
\newcommand{\mt}{\mathsf}
\newcommand{\md}{\mathcal}
\newcommand{\ld}{\left}
\newcommand{\rd}{\right}

\newcommand{\ip}{\int_{-\pi}^{\pi}}

\newcommand{\be}{\begin{equation}}
\newcommand{\ee}{\end{equation}}

\newcommand{\bem}{\begin{multline}}
\newcommand{\eem}{\end{multline}}

\newcommand{\bml}{\begin{multline*}}
\newcommand{\eml}{\end{multline*}}

\newcommand{\beg}{\begin{gather}}
\newcommand{\eeg}{\end{gather}}

\begin{document}

\title{Robust filtering of sequences with periodically stationary multiplicative seasonal  increments}

\author{
Maksym Luz\thanks {BNP Paribas Cardif, Kyiv, Ukraine, maksym.luz@gmail.com},
Mikhail Moklyachuk\thanks
{Department of Probability Theory, Statistics and Actuarial
Mathematics, Taras Shevchenko National University of Kyiv, Kyiv 01601, Ukraine, moklyachuk@gmail.com}
}

\date{\today}

\maketitle

\renewcommand{\abstractname}{Abstract}
\begin{abstract}
We study  stochastic sequences $\xi(k)$ with periodically stationary generalized multiple increments of fractional order which combines cyclostationary, multi-seasonal, integrated and fractionally integrated patterns. We solve the filtering problem for linear functionals
constructed from unobserved values of a stochastic sequence $\xi(k)$  based on observations with the periodically stationary noise sequence.
For sequences with known matrices of spectral densities, we obtain formulas for calculating values of the mean square errors and the spectral characteristics of the optimal estimates of the functionals.
Formulas that determine the least favorable spectral densities and minimax (robust) spectral
characteristics of the optimal linear estimates of the functionals
are proposed in the case where spectral densities of sequences
are not exactly known while some sets of admissible spectral densities are given.
\end{abstract}

\vspace{2ex}

\maketitle

\vspace{2ex}

\theoremstyle{plain}

\section{Introduction}

Non-stationary and long memory time series models are wildly used in different fields of economics, finance, climatology, air pollution, signal processing etc.
(see, for example, papers by Dudek and Hurd [9], Johansen and Nielsen  [22], Reisen et al. [42]).
A core example -- a general multiplicative model, or SARIMA $(p, d, q)\times(P, D, Q)_s$ -- was introduced in the book by Box and Jenkins [5].
It includes both  integrated and  seasonal factors:
\begin{equation}
\Psi (B^s) \psi (B) (1-B)^d (1-B^s)^Dx_t = \Theta (B^s) \theta (B)\eps_t,
\end{equation}
where $ \Psi (z) $ and $ \Theta (z) $ are two polynomials of degrees of $ P $ and $ Q $ respectively which have roots outside the unit circle.
The parameters $d$ and $D$ are allowed to be fractional. When $|d+D|<1/2$ and $|D|<1/2$, the process  (1) is stationary and invertible.
The paper  by Porter-Hudak [41] illustrates an application of a seasonal ARFIMA model to the analysis of the monetary aggregates used by U.S. Federal Reserve.
Another model of fractional integration is GARMA processes  described by the equation (see Gray,  Cheng and  Woodward [16])
\be
 (1-2uB+B^2) ^ dx_t = \eps_t,\quad |u|\leq1.
\label{GARMA_model}
\ee

For the resent results dedicated to the statistical inference for seasonal long-memory sequences,
we refer  to the paper by Tsai,  Rachinger and  Lin [46], who developed methods of estimation of parameters in case of measurement errors.
In their paper Baillie,   Kongcharoen and  Kapetanios [3] compared MLE and semiparametric estimation procedures for prediction problems based on ARFIMA models.
 Based on simulation study, they indicate better performance of  MLE predictor than the one  based on the two-step local Whittle estimation.
Hassler and  Pohle [19] (see also Hassler [18]) assess a predictive  performance of various methods of  forecasting  of inflation and return volatility time series and show strong evidences for models with a fractional integration component.

Another type of non-stationary processes are periodically correlated, or cyclostationary, processes introduced by Gladyshev [13], which belong to the class of processes with time-dependent spectrum and are widely used in signal processing and communications (see Napolitano [37] for a review of recent works on cyclostationarity and its applications).
Periodic time series are also considered as extension of seasonal models
[2, 4, 28, 38].

One of the fields of interests related to time series analysis is optimal filtering. It aims to remove the unobserved components, such as trends, seasonality or noise signal, from the observed data  [1, 6].

Methods of parameters estimations and filtering usually do not take into account the issues arising from real data, namely, the presence of outliers, measurement errors, incomplete information about the spectral, or model, structure etc. From this point of view, we see an increasing interest to robust methods of estimation that are reasonable in such cases (see
Reisen,  et al. [43], Solci at al. [45] for the examples of robust estimates of SARIMA and PAR models). Grenander [15], Hosoya [20], Kassam [24], Franke [10], Vastola and  Poor [47], Moklyachuk[32, 33] studied minimax  extrapolation, interpolation and filtering problems for stationary sequences and processes.

This article is dedicated to the robust filtering of stochastic sequences with periodically stationary long memory multiple seasonal increments, or sequences with periodically stationary general multiplicative (GM) increments, introduced by Luz and Moklyachuk [31]. In the recent years, models with multiple seasonal and periodic patterns are in scope of interest, see  Dudek [8], Gould et al. [14] and Reisen et al. [42], Hurd and Piparas [21]. This  motivates us to study the time series combining a periodic structure of the covariation function as well as multiple seasonal factors. The discussed problem is a natural continuation of the researches on minimax filtering of stationary vector-valued processes, periodically correlated processes and processes with stationary increments have been performed by  Moklyachuk and Masyutka [34],  Moklyachuk and  Golichenko (Dubovetska) [7],  Luz and Moklyachuk [29, 30]  respectively. We also mention the works [35,36] by Moklyachuk and Sidei, who derive minimax estimates of stationary processes from observations with missed values, and the work [27]  by Moklyachuk and  Kozak, who have studied interpolation problem for stochastic sequences with periodically stationary increments.

The article is organized as follows.
In Section 2, we recall definitions of
 generalized multiple (GM)  increment sequence $\chi_{\overline{\mu},\overline{s}}^{(d)}(\vec{\xi}(m))$ and stochastic
sequences $\xi(m)$ with periodically stationary (periodically correlated, cyclostationary) GM increments.
The spectral theory of vector-valued GM increment sequences is discussed.
Section 3 deals with the classical filtering problem for linear functionals $A\xi$ and $A_N\xi $ which are constructed from unobserved values of the sequence $\xi(m)$ when the spectral densities of the sequence $\xi(m)$ and a noise  sequence $\eta(m)$ are known.
Estimates are obtained by applying the Hilbert space projection technique to the vector sequence $\vec \xi(m)+ \vec \eta(m)$ with stationary GM  increments under the stationary noise sequence $\vec \eta(m)$ uncorrelated with $\vec \xi(m)$.
The case of non-stationary fractional integration is discussed as well.
Section 4 is dedicated to the minimax (robust) estimates in cases, where spectral densities of sequences are not exactly known
while some sets of admissible spectral densities are specified. We illustrate the proposed technique on the particular types of the sets, which are generalizations of the sets of admissible spectral densities described in a survey article by
Kassam and Poor [25] for stationary stochastic processes.

\section{Stochastic sequences with periodically stationary generalized multiple increments}\label{spectral_ theory}

\subsection{Preliminary notations and definitions}
 Consider a   stochastic sequence $\xi(m)$, $m\in\mathbb Z$, and a backward shift operator $B_{\mu}$   with the step $\mu\in
\mathbb Z$, such that $B_{\mu}\xi(m)=\xi(m-\mu)$; $B:=B_1$. Then $B_{\mu}^s=B_{\mu}B_{\mu}\cdot\ldots\cdot B_{\mu}$. Define an incremental operator $\chi_{\overline{\mu},\overline{s}}^{(d)}(B)
=(1-B_{\mu_1}^{s_1})^{d_1}(1-B_{\mu_2}^{s_2})^{d_2}\cdot\ldots\cdot(1-B_{\mu_r}^{s_r})^{d_r}$, where
$d:=d_1+d_2+\ldots+d_r$, $\overline{d}=(d_1,d_2,\ldots,d_r)\in (\mr N^*)^r$,
 $\overline{s}=(s_1,s_2,\ldots,s_r)\in (\mr N^*)^r$
and $\overline{\mu}=(\mu_1,\mu_2,\ldots,\mu_r)\in (\mr N^*)^r$ or $\in (\mr Z\setminus\mr N)^r$. Here $\mr N^*=\mr N\setminus\{0\}$.

Within the article, $\delta_{lp}$ denotes Kronecker symbols, ${n \choose l}=\frac{n!}{l!(n-l)!}$.

\begin{ozn}\label{def_multiplicative_Pryrist}
For a stochastic sequence $\xi(m)$, $m\in\mathbb Z$, the
sequence
\begin{eqnarray}
\nonumber
\chi_{\overline{\mu},\overline{s}}^{(d)}(\xi(m))&:=&\chi_{\overline{\mu},\overline{s}}^{(d)}(B)\xi(m)
=(1-B_{\mu_1}^{s_1})^{d_1}(1-B_{\mu_2}^{s_2})^{d_2}\cdots(1-B_{\mu_r}^{s_r})^{d_r}\xi(m)
\\&=&\sum_{l_1=0}^{d_1}\ldots \sum_{l_r=0}^{d_r}(-1)^{l_1+\dots+ l_r}{d_1 \choose l_1}\cdots{d_r \choose l_r}\xi(m-\mu_1s_1l_1-\cdots-\mu_rs_rl_r)
\label{GM_Pryrist}
\end{eqnarray}
is called \emph{stochastic  generalized multiple (GM)  increment sequence} of differentiation   order
$d$
with a fixed seasonal  vector $\overline{s}\in (\mr N^*)^r$
and a varying step $\overline{\mu}\in (\mr N^*)^r$ or $\in (\mr Z\setminus\mr N)^r$.
\end{ozn}

 The  multiplicative increment operator $\chi_{\overline{\mu},\overline{s}}^{(d)}(B)$ admits the representation
\[
\chi_{\overline{\mu},\overline{s}}^{(d)}(B)
=\prod_{i=1}^r(1-B_{\mu_i}^{s_i})^{d_i}
=\sum_{k=0}^{n(\gamma)}e_{\gamma}(k)B^k,
\]
where $n(\gamma):=\sum_{i=1}^r\mu_is_id_i$. The explicit representation of the coefficients $e_{\gamma}(k)$ is given in [31].

\begin{ozn}
\label{oznStPryrostu}
A stochastic GM increment sequence $\chi_{\overline{\mu},\overline{s}}^{(d)}(\xi(m))$  is called  wide sense
stationary if the mathematical expectations
\begin{eqnarray*}
\mt E\chi_{\overline{\mu},\overline{s}}^{(d)}(\xi(m_0))& = &c^{(d)}_{\overline{s}}(\overline{\mu}),
\\
\mt E\chi_{\overline{\mu}_1,\overline{s}}^{(d)}(\xi(m_0+m))\chi_{\overline{\mu}_2,\overline{s}}^{(d)}(\xi(m_0))
& = & D^{(d)}_{\overline{s}}(m;\overline{\mu}_1,\overline{\mu}_2)
\end{eqnarray*}
exist for all $m_0,m,\overline{\mu},\overline{\mu}_1,\overline{\mu}_2$ and do not depend on $m_0$.
The function $c^{(d)}_{\overline{s}}(\overline{\mu})$ is called mean value  and the function $D^{(d)}_{\overline{s}}(m;\overline{\mu}_1,\overline{\mu}_2)$ is
called structural function of the stationary GM increment sequence (of a stochastic sequence with stationary GM increments).
\\
The stochastic sequence $\xi(m)$, $m\in\mathbb   Z$
determining the stationary GM increment sequence
$\chi_{\overline{\mu},\overline{s}}^{(d)}(\xi(m))$ by   (3) is called stochastic
sequence with stationary GM increments (or GM increment sequence of order $d$).
\end{ozn}

\begin{zau}
The particular case of the one-pattern increment sequence $\chi_{\mu,1}^{(n)}(\xi(m)):=\xi^{(n)}(m,\mu)=(1-B_{\mu})^n\xi(m)$ and the continues time increment process $\xi^{(n)}(t,\tau)=(1-B_{\tau})^n\xi(t)$ were studied in [40, 48, 49].
\end{zau}

\subsection{Definition and spectral representation of periodically stationary GM increment}

In this section, we present definition, justification and a brief review of the spectral theory of stochastic sequences with periodically stationary multiple seasonal increments.

\begin{ozn}
\label{OznPeriodProc}
A stochastic sequence $\xi(m)$, $m\in\mathbb Z$ is called stochastic
sequence with periodically stationary (periodically correlated) GM increments with period $T$ if the mathematical expectations
\begin{eqnarray*}
\mt E\chi_{\overline{\mu},T\overline{s}}^{(d)}(\xi(m+T)) & = & \mt E\chi_{\overline{\mu},T\overline{s}}^{(d)}(\xi(m))=c^{(d)}_{T\overline{s}}(m,\overline{\mu}),
\\
\mt E\chi_{\overline{\mu}_1,T\overline{s}}^{(d)}(\xi(m+T))\chi_{\overline{\mu}_2,T\overline{s}}^{(d)}(\xi(k+T))
& = & D^{(d)}_{T\overline{s}}(m+T,k+T;\overline{\mu}_1,\overline{\mu}_2)
= D^{(d)}_{T\overline{s}}(m,k;\overline{\mu}_1,\overline{\mu}_2)
\end{eqnarray*}
exist for every  $m,k,\overline{\mu}_1,\overline{\mu}_2$ and  $T>0$ is the least integer for which these equalities hold.
\end{ozn}

It follows from  Definition 2.3 that the sequence
\begin{equation}
\label{PerehidXi}
\xi_{p}(m)=\xi(mT+p-1), \quad p=1,2,\dots,T; \quad m\in\mathbb Z
\end{equation}
forms a vector-valued sequence
$\vec{\xi}(m)=\left\{\xi_{p}(m)\right\}_{p=1,2,\dots,T}, m\in\mathbb Z$
with stationary GM increments as follows:
\begin{eqnarray*}
\chi_{\overline{\mu},\overline{s}}^{(d)}(\xi_p(m))&=&\sum_{l_1=0}^{d_1}\ldots \sum_{l_r=0}^{d_r}(-1)^{l_1+\ldots+ l_r}{d_1 \choose l_1}\cdot\ldots\cdot{d_r \choose l_r}\xi_p(m-\mu_1s_1l_1-\ldots-\mu_rs_rl_r)
\\
&=&\sum_{l_1=0}^{d_1}\ldots \sum_{l_r=0}^{d_r}(-1)^{l_1+\ldots+ l_r}{d_1 \choose l_1}\cdot\ldots\cdot{d_r \choose l_r}\xi((m-\mu_1s_1l_1-\ldots-\mu_rs_rl_r)T+p-1)
\\
&=&\chi_{\overline{\mu},T\overline{s}}^{(d)}(\xi(mT+p-1)),\quad p=1,2,\dots,T,
\end{eqnarray*}
where $\chi_{\overline{\mu},\overline{s}}^{(d)}(\xi_p(m))$ is the GM increment of the $p$-th component of the vector-valued sequence $\vec{\xi}(m)$.

\begin{pry}
Define a periodic seasonal autoregressive integrated moving average model (PSARIMA) $\{X_m,m\in \mr Z\}$,  with multiple seasonal patterns by relation
\[
\phi_m(B)(1-B^T)^d\prod_{i=1}^r\Phi_{i,m}(B)(1-B^{Ts_i})^{d_i}X_m
=\theta_m(B) \prod_{i=1}^r\Theta_{i,m}(B) \eps_m,
\]
where all  polynomials $\phi_m(z)$, $\theta_m(z)$, $\Phi_{i,m}(z)$, $\Theta_{i,m}(z)$ are $T$-periodic by parameter  $m$ functions, $1<s_1<\ldots<s_r$.
Define
\begin{eqnarray*}
\Phi_m(z)& := &\phi_m(z)\prod_{i=1}^r\Phi_{i,m}(z)=\sum_{k=0}^{q_1}\Phi_m(k)z^k,
\\
\Theta_m(z)& := &\theta_m(z) \prod_{i=1}^r\Theta_{i,m}(z)=\sum_{k=0}^{q_2}\Theta_m(k)z^k
\end{eqnarray*}
 and put $\Phi_m(k)=0$ for $k>q_1$, $\Theta_m(k)=0$ for $k>q_2$. Then the increment sequence $$Y_m=(1-B^T)^d\prod_{i=1}^r(1-B^{Ts_i})^{d_i}X_m$$
 is periodically stationary and allows a stationary vector representation
$$\me Y_m=(1-B)^d\prod_{i=1}^r(1-B^{s_i})^{d_i}\me X_m$$ with
$$\me Y_m=(Y_{mT},Y_{mT+1},\ldots, Y_{mT+T-1})^{\top}, \quad\me X_m=(X_{mT},X_{mT+1},\ldots, X_{mT+T-1})^{\top},$$ $\boldsymbol{\eps}_m=(\eps_{mT},\eps_{mT+1},\ldots, \eps_{mT+T-1})^{\top}$.
We can write the relation
\[
\me  \Pi \me Y_m+\sum_{l=1}^{q_1^*}\me \Pi_l\me Y_{m-l}=\me \Xi \boldsymbol{\eps}_m+\sum_{l=1}^{q_2^*}\me \Xi_l\boldsymbol{\eps}_{m-l},
\]
where $\me \Pi(k,j)=\Phi_k(k-j)$, $\me \Xi(k,j)=\Theta_k(k-j)$ for $k\geq j$, $\me \Pi(k,j)=0$, $\me \Xi(k,j)=0$ otherwise.
$\me \Pi_l(k,j)=\Phi_k(lT+k-j)$, $\me \Xi_l(k,j)=\Theta_k(lT+k-j)$ [9], provided
$det(\me \Pi  +\sum_{l=1}^{q_1^*}\me \Pi_lz^l)\neq 0$ for $|z|\leq 1$ [17].
A GM increment sequence is defined as
$$\chi_{\overline{\mu},\overline{s}}^{(d)}(\me X_m)=(1-B^{\mu_0})^d\prod_{i=1}^r(1-B^{s_i\mu_i})^{d_i}\me X_m,\,\, m\in \mr Z.$$
\end{pry}

The following theorem describes the spectral structure of the GM increment  [23], [31].

\begin{thm}\label{thm1}
1. The mean value and the structural function
 of the vector-valued stochastic stationary
GM increment sequence $\chi_{\overline{\mu},\overline{s}}^{(d)}(\vec{\xi}(m))$ can be represented in the form
\begin{eqnarray}
\label{serFnaR_vec}
c^{(d)}_{ \overline{s}}(\overline{\mu})& = &c\prod_{i=1}^r\mu_i^{d_i},
\\
\label{strFnaR_vec}
 D^{(d)}_{\overline{s}}(m;\overline{\mu}_1,\overline{\mu}_2)& = &\int_{-\pi}^{\pi}e^{i\lambda
m} \chi_{\overline{\mu}_1}^{(d)}(e^{-i\lambda})\chi_{\overline{\mu}_2}^{(d)}(e^{i\lambda})\frac{1}
{|\beta^{(d)}(i\lambda)|^2}dF(\lambda),
\end{eqnarray}
where
\[\chi_{\overline{\mu}}^{(d)}(e^{-i\lambda})=\prod_{j=1}^r(1-e^{-i\lambda\mu_js_j})^{d_j}, \quad \beta^{(d)}(i\lambda)= \prod_{j=1}^r\prod_{k_j=-[s_j/2]}^{[s_j/2]}(i\lambda-2\pi i k_j/s_j)^{d_j},
\]
 $c$ is a vector, $F(\lambda)$ is the matrix-valued spectral function of the stationary stochastic sequence $\chi_{\overline{\mu},\overline{s}}^{(d)}(\vec{\xi}(m))$. The vector $c$
and the matrix-valued function $F(\lambda)$ are determined uniquely by the GM
increment sequence $ \chi_{\overline{\mu},\overline{s}}^{(d)}(\vec \xi(m))$.

2. The stationary GM increment sequence $\chi_{\overline{\mu},\overline{s}}^{(d)}(\vec{\xi}(m))$ admits the spectral representation
\begin{equation}
\label{SpectrPred_vec}
\chi_{\overline{\mu},\overline{s}}^{(d)}(\vec{\xi}(m))
=\int_{-\pi}^{\pi}e^{im\lambda}\chi_{\overline{\mu}}^{(d)}(e^{-i\lambda})\frac{1}{\beta^{(d)}(i\lambda)}d\vec{Z}_{\xi^{(d)}}(\lambda),
\end{equation}
where $d\vec{Z}_{\xi^{(d)}}(\lambda)=\{Z_{ p}(\lambda)\}_{p=1}^{T}$ is a (vector-valued) stochastic process with uncorrelated increments on $[-\pi,\pi)$ connected with the spectral function $F(\lambda)$ by
the relation
\[
 \mt E(Z_{p}(\lambda_2)-Z_{p}(\lambda_1))(\overline{ Z_{q}(\lambda_2)-Z_{q}(\lambda_1)})
 =F_{pq}(\lambda_2)-F_{pq}(\lambda_1),\]
 \[  -\pi\leq \lambda_1<\lambda_2<\pi,\quad p,q=1,2,\dots,T.
\]
\end{thm}

Consider another vector stochastic sequence with the stationary GM
increments $\vec \zeta (m)=\vec \xi(m)+\vec \eta(m)$, where $\vec\eta(m)$ is a vector stationary stochastic sequence, uncorrelated with $\vec\xi(m)$, with a spectral representation
\[
 \vec\eta(m)=\int_{-\pi}^{\pi}e^{i\lambda m}dZ_{\eta}(\lambda),\]
 $Z_{\eta}(\lambda)=\{Z_{\eta,p}(\lambda)\}_{p=1}^T$, $[-\pi,\pi)$, is a stochastic process with uncorrelated increments, that corresponds to the spectral function $G(\lambda)$ [17].
The stochastic stationary GM increment $\chi_{\overline{\mu},\overline{s}}^{(d)}(\vec{\zeta}(m))$ allows the spectral representation
\begin{eqnarray*}
 \chi_{\overline{\mu},\overline{s}}^{(d)}(\vec{\zeta}(m))&=&\int_{-\pi}^{\pi}e^{i\lambda m}\frac{\chi_{\overline{\mu}}^{(d)}(e^{-i\lambda})}{\beta^{(d)}(i\lambda)}
 dZ_{\xi^{(n)}}(\lambda)
  +\int_{-\pi}^{\pi}e^{i\lambda m}\chi_{\overline{\mu}}^{(d)}(e^{-i\lambda}) dZ_{\eta }(\lambda),
 \end{eqnarray*}
while $dZ_{\eta }(\lambda)=(\beta^{(d)}(i\lambda))^{-1} dZ_{\eta^{(n)}}(\lambda)$,
$\lambda\in[-\pi,\pi)$. Therefore, in the case where the spectral functions $F(\lambda)$ and $G(\lambda)$ have the spectral densities $f(\lambda)$ and $g(\lambda)$, the spectral density $p(\lambda)=\{p_{ij}(\lambda)\}_{i,j=1}^{T}$ of the stochastic sequence $\vec \zeta(m)$ is determined by the formula
\[
 p(\lambda)=f(\lambda)+|\beta^{(d)}(i\lambda)|^2g(\lambda).\]

\subsection{GM fractional increments}\label{fractional_extrapolation}

In the previous subsection, we describe    the  GM increment sequence $\chi_{\overline{\mu},\overline{s}}^{(d)}(\vec{\xi}(m))$ of  the positive integer orders $(d_1,\ldots,d_r)$. Here we consider the case of possibly fractional increment orders $d_i$.

Within the subsection, we put the step $\overline{\mu}=(1,1,\ldots,1)$. Represent the increment operator $\chi_{\overline{s}}^{(d)}(B)$  in the form
\be\label{FM_increment}
\chi_{\overline{s}}^{(R+D)}(B)=(1-B)^{R_0+D_0}\prod_{j=1}^r(1-B^{s_j})^{R_j+D_j},
\ee
where $(1-B)^{R_0+D_0}$ is an integrating component, $R_j$, $j=0,1,\ldots, r$, are non-negative integer numbers, $1<s_1<\ldots<s_r$.  Below we describe a representations $d_j=R_j+D_j$, $j=0,1,\ldots, r$, of the increment orders $d_j$ by stating  conditions on the fractional parts $D_j$, such that the increment sequence $\vec y(m):=(1-B)^{R_0}\prod_{j=1}^r(1-B^{s_j})^{R_i}\vec{\xi}(m)$ is  a stationary  fractionally integrated seasonal stochastic  sequence.
For example, in case of single  increment pattern $(1-B^{s^*})^{R^*+D^*}$, this condition is $|D^*|<1/2$.

A  sequence $\chi_{\overline{s}}^{(R+D)}(\vec \xi(m))$ is called  \emph{a fractional multiple (FM) increment sequence}.

Consider a generating function of the Gegenbauer polinomial:
\[
(1-2 u B+B^2)^{-d}=\sum_{n=0}^{\infty}C_n^{(d)}(u)B^n,
\]
where
\[
C_n^{(d)}(u)=\sum_{k=0}^{[n/2]}\frac{(-1)^k(2u)^{n-2k}\Gamma(d-k+n)}{k!(n-2k)!\Gamma(d)}.
\]

The following lemma and  theorem hold true [31].

\begin{lema}\label{frac_incr_2}
Define the sets $\md M_j=\{\nu_{k_j}=2\pi k_j/s_j: k_j=0,1,\ldots, [s_j/2]\}$, $j=0,1,\ldots, r$, and the set $\md M=\bigcup_{j=0}^r \md M_j$. Then the increment operator $\chi_{\overline{s}}^{(D)}(B):=(1-B)^{D_0}\prod_{j=1}^r(1-B^{s_j})^{D_j}$ admits a representation
\begin{eqnarray*}
\chi_{\overline{s}}^{(D)}(B)
& = &\prod_{\nu \in \md M}(1-2\cos \nu B+B^2)^{\widetilde{D}_{\nu}}
\\
& = &(1-B)^{D_0+D_1+\ldots+D_r}(1+B)^{D_{\pi}}\prod_{\nu \in \md M\setminus\{0,\pi\}}(1-2\cos \nu B+B^2)^{D_{\nu}}
\\
& = &\ld(\sum_{m=0}^{\infty}G^+_{k^*}(m)B^m\rd)^{-1}=\sum_{m=0}^{\infty}G^-_{k^*}(m)B^m,
\end{eqnarray*}
where
\begin{eqnarray}
\label{Gegenbauer_GI+}
G^+_{k^*}(m)& = &\sum_{0\leq n_1,\ldots,n_{k^*}\leq m, n_1+\ldots+n_{k^*}=m}\prod_{\nu \in \md M}C_{n_{\nu}}^{(\widetilde{D}_{\nu})}(\cos\nu),
\\
\label{Gegenbauer_GI-}
G^-_{k^*}(m)& = &\sum_{0\leq n_1,\ldots,n_{k^*}\leq m, n_1+\ldots+n_{k^*}=m}\prod_{\nu \in \md M}C_{n_{\nu}}^{(-\widetilde{D}_{\nu})}(\cos\nu).
\end{eqnarray}
 $k^*=|\md M|$,  $D_{\nu}=\sum_{j=0}^rD_j \mr I \{\nu\in \md M_j\}$, $\widetilde{D}_{\nu}=D_{\nu}$ for $\nu \in \md M\setminus\{0,\pi\}$, $\widetilde{D}_{\nu}=D_{\nu}/2$ for $\nu=0$ and $\nu=\pi$.
\end{lema}

\begin{thm}\label{thm_frac}
Assume that for a stochastic vector sequence $\vec \xi(m)$ and fractional differencing orders $d_j=R_j+D_j$, $j=0,1,\ldots, r$, the FM increment sequence $\chi_{\overline{1},\overline{s}}^{(R+D)}(\vec \xi(m))$ generated by increment operator (8)  is a stationary sequence with a bounded from zero and infinity spectral density $\widetilde{f}_{\overline{1}}(\lambda)$. Then for the non-negative integer numbers $R_j$, $j=0,1,\ldots, r$, the GM increment sequence $\chi_{\overline{1},\overline{s}}^{(R)}(\vec \xi(m))$    is stationary if $-1/2< D_{\nu}<1/2$ for all $\nu\in \md M$, where $D_{\nu}$ are defined by real numbers $D_j$, $j=0,1,\ldots, r$,
in Lemma 2.1, and it is long memory if $0< D_{\nu}<1/2$ for at least one $\nu\in \md M$, and invertible if $-1/2< D_{\nu}<0$. The spectral density $f(\lambda)$ of the stationary GM increment sequence $\chi_{\overline{1},\overline{s}}^{(R)}(\vec \xi(m))$ admits a representation
\[
 f(\lambda)=|\beta^{(R)}(i\lambda)|^2 \ld|\chi_{\overline{1}}^{(R)}(e^{-i\lambda})\rd|^{-2}\ld|\chi_{\overline{1}}^{(D)}(e^{-i\lambda})\rd|^{-2} \widetilde{f}_{\overline{1}}(\lambda)=:\ld|\chi_{\overline{1}}^{(D)}(e^{-i\lambda})\rd|^{-2} \widetilde{f} (\lambda),
  \]
  where
  \begin{eqnarray*}
  \ld|\chi_{\overline{1}}^{(D)}(e^{-i\lambda})\rd|^{-2}& = &\ld|\sum_{m=0}^{\infty}G^+_{k^*}(m) e^{-i\lambda m}\rd|^2=\ld|\sum_{m=0}^{\infty}G^-_{k^*}(m) e^{-i\lambda m}\rd|^{-2}
\\
& = &\prod_{\nu \in \md M}\ld|(e^{-i\nu}-e^{i\lambda})(e^{i\nu}-e^{i\lambda})\rd|^{-2\widetilde{D}_{\nu}},
  \end{eqnarray*}
 coefficients $G^+_{k^*}(m)$, $G^-_{k^*}(m)$ are defined in (9), (10).
\end{thm}

The spectral density $f(\lambda)$ and the structural function $D^{(R)}_{ \overline{s}}(m,\overline{1},\overline{1})$ of a stationary GM increment sequence $\chi_{\overline{1},\overline{s}}^{(R)}(\vec \xi(m))$ exhibit the following behavior for the constant matrices $C$ and $K$:
\begin{itemize}
\item   $|\beta^{(R)}(i\lambda)|^{-2} |\chi_{\overline{1}}^{(R)}(e^{-i\lambda})|^2f(\lambda)\sim C|\nu-\lambda|^{-2\widetilde{D}_{\nu}}$   when $\lambda\to \nu$, $\nu\in \mathcal{M}$ (for properties of eigenvalues of generalized fractional process, we refer to Palma and Bondon [39])

\item $D^{(R)}_{ \overline{s}}(m,\overline{1},\overline{1})\sim K\sum_{\nu\in \mathcal{M}:\widetilde{D}_{\nu}>0}|m|^{2\widetilde{D}_{\nu}-1}\cos (m\nu)$,   as $m\to\infty$, see Giraitis and Leipus [12].
\end{itemize}

\begin{pry}
1. For the increment operator $(1-B)^{R_0+D_0}(1-B^2)^{R_1+D_1}$, $\md M_0=\{0\}$, $\md M_1=\{0,\pi\}$, $\md M=\{0,\pi\}$, the Gegenbauer representation  is $(1-B)^{D_0+D_1}(1+B)^{D_1}$.
The stationarity conditions are: $|D|=|D_0+D_1|<1/2$, $|D_{\pi}|=|D_1|<1/2$.

2. For the increment operator  $(1-B^2)^{R_1+D_1}(1-B^3)^{R_2+D_2}$,  $\md M_0=\{0,\pi\}$, $\md M_1=\{0,2\pi/3\}$, $\md M=\{0,2\pi/3,\pi\}$,the Gegenbauer representation   is $(1-B)^{D_1+D_2}(1-2\cos(2\pi/3)B+B^2)^{D_2}(1+B)^{D_1}$.
The stationarity conditions are:$|D|=|D_1+D_2|<1/2$, $|D_{2\pi/3}|=|D_2|<1/2$, $|D_{\pi}|=|D_1|<1/2$.

3. For the increment operator  $(1-B^2)^{R_1+D_1}(1-B^4)^{R_2+D_2}$, $\md M_0=\{0,\pi\}$, $\md M_1=\{0,\pi/2,\pi\}$, $\md M=\{0,\pi/2,\pi\}$, the Gegenbauer representation   is $(1-B)^{D_1+D_2}(1+B^2)^{D_2}(1+B)^{D_1+D_2}$.
The stationarity conditions are: $|D|=|D_{\pi}|=|D_1+D_2|<1/2$, $|D_{\pi/2}|=|D_2|<1/2$.
\end{pry}

\section{Hilbert space projection method of filtering}\label{classical_filtering}

\subsection{Filtering of vector-valued stationary GM increment} \label{classical_filtering_vector}

Consider a vector-valued stochastic sequence with stationary GM increments $\vec{\xi}(m)$ constructed from transformation (4) and a vector-valued stationary stochastic sequence $\vec\eta(m)$ uncorrelated with the sequence $\vec\eta(m)$.
Let the stationary GM increment sequence $\chi_{\overline{\mu},\overline{s}}^{(d)}(\vec{\xi}(m))=\{\chi_{\overline{\mu},\overline{s}}^{(d)}(\xi_p(m))\}_{p=1}^{T}$
and the stationary stochastic sequence $\vec\eta(m)$ have absolutely continuous spectral function $F(\lambda)$ and $G(\lambda)$ with the spectral densities $f(\lambda)=\{f_{ij}(\lambda)\}_{i,j=1}^{T}$ and $g(\lambda)=\{g_{ij}(\lambda)\}_{i,j=1}^{T}$ respectively.

Without loss of generality we will assume that $ \mathsf{E}\chi_{\overline{\mu},\overline{s}}^{(d)}(\vec{\xi}(m))=0$, $\mathsf{E}\vec\eta(m)=0$ and $\overline{\mu}>\overline{0}$.

\textbf{Filtering problem.} Consider the problem of mean square optimal linear estimation of the functional
\begin{equation}
A\vec{\xi}=\sum_{k=0}^{\infty}(\vec{a}(k))^{\top}\vec{\xi}(-k),
\end{equation}
which depends on unobserved values of a stochastic sequence $\vec{\xi}(k)=\{\xi_{p}(k)\}_{p=1}^{T}$ with stationary GM
increments. Estimates are based on observations of the sequence $\vec\zeta(k)=\vec\xi(k)+\vec\eta(k)$ at points $k=0,-1,-2,\ldots$.

We will suppose that the following conditions are satisfied:
\begin{itemize}
\item  conditions on coefficients $\vec{a}(k)=\{a_{p}(k)\}_{p=1}^{T}$, $k\geq0$
 \be\label{umova
na a_f_st.n_d}
\sum_{k=0}^{\infty}\|\vec{a}(k)\|<\infty,\quad
\sum_{k=0}^{\infty}(k+1)\|\vec{a}(k)\|^{2}<\infty,\ee

\item  \emph{the minimality condition} on the spectral densities $f(\lambda)$ and $g(\lambda)$
\be
 \ip \text{Tr}\left[ \frac{|\beta^{(d)}(i\lambda)|^2}{|\chi_{\overline{\mu}}^{(d)}(e^{-i\lambda})|^2}\ld(f(\lambda)+|\beta^{(d)}(i\lambda)|^2 g(\lambda)\rd)^{-1}\right]
 d\lambda<\infty.
\label{umova11_f_st.n_d}
\ee
\end{itemize}
The latter condition  is the necessary and sufficient one under which the mean square errors of estimates of functional  $A\vec\xi$ is not equal to $0$.

We apply the  Hilbert space estimation technique proposed by Kolmogorov [26] which can be described as a $3$-stage procedure:
(i) define a target element (to be estimated) of the space $H=L_2(\Omega, \mathcal{F},\mt P)$ of random variables $\gamma$ which have zero mean values and finite variances, $\mt E\gamma=0$, $\mt E|\gamma|^2<\infty$, endowed with the inner product $\langle \gamma_1,\gamma_2\rangle={\mt E}{\gamma_1\overline{\gamma_2}}$,
(ii) define a subspace of $H$ generated by observations,
(iii) find an estimate of the target element as an orthogonal projection on the defined subspace.

\emph{Stage (i)}.
Represent the functional $A\vec \xi$ as
$A\vec \xi=A\vec \zeta-A\vec \eta$, where
\[
 A\vec\zeta=\sum_{k=0}^{\infty}(\vec a(k))^{\top}\vec \zeta(-k),\quad A\vec\eta=\sum_{k=0}^{\infty}(\vec a(k))^{\top}\vec \eta(-k).\]
Note, that under conditions (12), the functional $A\vec\eta$ belongs to the space $H$.
To find the estimate of the functional $A\vec\xi$, it is sufficient to find the estimate of the functional $A\vec\eta$:
\be\label{mainformula_f_st.n_d}
 \widehat{A}\vec\xi=A\vec\zeta-\widehat{A}\vec\eta.\ee
 Since the functional $A\vec\zeta$ depends on values of the stochastic sequence $\zeta(k)$ which is observed,  the mean square errors of the optimal estimates $\widehat{A}\vec\xi$ and $\widehat{A}\vec\eta$ are connected by relation
 \begin{eqnarray*}
 \Delta\ld(f,g;\widehat{A}\vec\xi\rd)&=&\mt E \ld|A\vec\xi-\widehat{A}\vec\xi\rd|^2= \mt
 E\ld|A\vec\zeta-A\vec\eta-A\vec\zeta+\widehat{A}\vec\eta\rd|^2
 \\&=&\mt E\ld|A\vec\eta-\widehat{A}\vec\eta\rd|^2=\Delta\ld(f,g;\widehat{A}\vec\eta\rd).\end{eqnarray*}
Relation (14) implies that any linear estimate $\widehat{A}\vec\xi$ of the functional $A\vec\xi$
allows the representation
\be
 \label{otsinka A_f_st.n_d}
 \widehat{A}\vec\xi=\sum_{k=0}^{\infty}(\vec a(k))^{\top}(\vec\xi(-k)+\vec\eta(-k))-\ip
 (\vec h_{\overline{\mu}}(\lambda))^{\top}dZ_{\xi^{(n)}+\eta^{(n)}}(\lambda), \ee
where
$ \vec h_{\overline{\mu}}(\lambda)=\{ h_p(\lambda)\}_{p=1}^T$ is the spectral characteristic of the estimate
$\widehat{A}\vec\xi$.

\emph{Stage (ii).} The set of the observed GM increments $\{\chi_{\overline{\mu},\overline{s}}^{(d)}(\vec{\xi}(k))+\chi_{\overline{\mu},\overline{s}}^{(d)}(\vec{\eta}(k)):k\leq0\}$, $\overline{\mu}>\vec 0$ generates   a closed linear subspace $H^{0}(\xi^{(d)}+\eta^{(d)})$ of the Hilbert space $H=L_2(\Omega, \mathfrak{F},
\mt P)$.
   The functions
\[
 e^{i\lambda k}\chi_{\overline{\mu}}^{(d)}(e^{-i\lambda})\frac{1}{\beta^{(d)}(i\lambda)}\vec\delta_l,\quad \vec\delta_l=\{\delta_{lp}\}_{p=1}^T,\,l=1,\ldots,T,\, k\leq0,\]
generate a closed linear subspace $L_2^{0}(p)$ of the Hilbert space
$L_2(p)$.
Then the relation
\[
 \chi_{\overline{\mu},\overline{s}}^{(d)}(\vec{\xi}(k))+\chi_{\overline{\mu},\overline{s}}^{(d)}(\vec{\eta}(k))=\ip e^{i\lambda
 k}\chi_{\overline{\mu}}^{(d)}(e^{-i\lambda})\dfrac{1}{\beta^{(d)}(i\lambda)}dZ_{\xi^{(d)}+\eta^{(d)}}(\lambda)\]
implies a map between elements
\[
e^{i\lambda k}\chi_{\overline{\mu}}^{(d)}(e^{-i\lambda})\dfrac{1}{\beta^{(d)}(i\lambda)}\vec \delta_l\]
from the space $L_2^{0}(p)$
and elements $\chi_{\overline{\mu},\overline{s}}^{(d)}(\vec{\xi}(k))+\chi_{\overline{\mu},\overline{s}}^{(d)}(\vec{\eta}(k))$ from the space
$H^{0}(\xi^{(d)}+\eta^{(d)})$ respectively.

\emph{Stage (iii).} The estimate
$\widehat{A}\vec\eta$ is found as an orthogonal projection of the functional $A\vec\eta$ from the space $H=L_2(\Omega, \mathfrak{F},
\mt P)$ on the subspace $H^{0}(\xi^{(d)}+\eta^{(d)})$. This projection  is characterized by the following conditions:

1) $ \widehat{A}\vec\eta\in H^{0}(\xi^{(d)}+\eta^{(d)}) $;

2) $(A\vec\eta-\widehat{A}\vec\eta)\perp
H^{0}(\xi^{(d)}+\eta^{(d)})$.

Define the following matrix-valued Fourier coefficients:
\[
 S^{\overline{\mu}}_{k,j}=\frac{1}{2\pi}\ip e^{-i\lambda(k+j)}
 \frac{|\beta^{(d)}(i\lambda)|^2}{|\chi_{\overline{\mu}}^{(d)}(e^{-i\lambda})|^2}g(\lambda)p^{-1}(\lambda)d\lambda,
 \quad k\geq0,j\geq-n(\gamma),\]
\[
 P_{k,j}^{\overline{\mu}}=\frac{1}{2\pi}\ip e^{i\lambda (j-k)}
 \dfrac{|\beta^{(d)}(i\lambda)|^2}{|\chi_{\overline{\mu}}^{(d)}(e^{-i\lambda})|^2}p^{-1}(\lambda)
 d\lambda,\quad k,j\geq0,\]
\[
 Q_{k,j}=\frac{1}{2\pi}\ip
 e^{i\lambda(j-k)} f(\lambda)g(\lambda)p^{-1}(\lambda)d\lambda,\quad k,j\geq0.\]
Define the vectors $$\me a=((\vec a(0))^{\top},(\vec a(1))^{\top},(\vec a(2))^{\top},\ldots)^{\top},\quad  \me a_{\overline{\mu}}=((\vec a_{\overline{\mu}}(0))^{\top},
(\vec a_{\overline{\mu}}(1))^{\top},(\vec a_{\overline{\mu}}(2))^{\top},\ldots)^{\top},$$ where the  coefficients $ a_{\overline{\mu}}(k)=a_{-\mu}(k-n(\gamma))$, $k\geq0$,
\be\label{coeff a_mu_f_st.n_d}
 \vec a_{-\overline{\mu}}(m)=\sum_{l=\max\ld\{m,0\rd\}}^{m+n(\gamma)}e_{\gamma}(l-m)\vec a(l),\quad m\geq-n(\gamma).\ee
Denote by $\me P_{\overline{\mu}}$,  $\me S_{\overline{\mu}}$ and $\me Q$  matrices with the matrix entries $(\me P_{\overline{\mu}})_{l,k}=P_{l,k}^{\overline{\mu}}$,  $(\me S_{\overline{\mu}})_{l, k} =S^{\overline{\mu}}_{l+1,k-n(\gamma)}$, $(\me Q)_{l,k}=Q_{l,k}$,  $l,k\geq0$.

\begin{thm}\label{thm_est_A}
A solution $\widehat{A}\vec\xi$ to the filtering problem for a vector-valued stochastic sequence $\vec{\xi}(m)$ with
stationary  GM increments under conditions  (12) and (13) is  calculated by formula (15).
The spectral characteristic $\vec h_{\overline{\mu}}(\lambda)$ and the value of the mean square error $\Delta(f,g;\widehat{A}\vec\xi)$ are calculated by the formulas
\be\label{spectr A_f_st.n_d}
 (\vec h_{\overline{\mu}}(\lambda))^{\top}=\ld[\chi_{\overline{\mu}}^{(d)}(e^{i\lambda})(A(e^{-i\lambda }))^{\top}
  g(\lambda)-(C_{\overline{\mu}}(e^{i\lambda}))^{\top}\rd]p^{-1}(\lambda)\frac{\overline{\beta^{(d)}(i\lambda)}
}{\chi_{\overline{\mu}}^{(d)}(e^{i\lambda})},\ee
\[A(e^{-i\lambda })=\sum_{k=0}^{\infty}\vec a(k)e^{-i\lambda k},
\quad
 C_{\overline{\mu}}(e^{i\lambda})=\sum_{k=0}^{\infty} \ld(\me
 P_{\overline{\mu}}^{-1}\me S_{\overline{\mu}}\me a_{\overline{\mu}}\rd)_ke^{i\lambda
 (k+1)},\]
 and
\begin{eqnarray}
 \notag && \Delta\ld(f,g;\widehat{A}\vec\xi\rd)=\Delta\ld(f,g;\widehat{A}\vec\eta\rd)= \mt E\ld|A\vec\eta-\widehat{A}\vec\eta\rd|^2=
 \\
 \notag&=&\frac{1}{2\pi}\ip
 \frac{1}{|\chi_{\overline{\mu}}^{(d)}(e^{-i\lambda})|^2}
\ld(\chi_{\overline{\mu}}^{(d)}(e^{i\lambda})(A(e^{-i\lambda }))^{\top} f(\lambda)+|\beta^{(d)}(i\lambda)|^2(C_{\overline{\mu}}(e^{i\lambda}))^{\top}\rd)
\\
\notag &\quad & \times p^{-1}(\lambda) g(\lambda)p^{-1}(\lambda)\ld(\chi_{\overline{\mu}}^{(d)}(e^{-i\lambda})f(\lambda) A(e^{i\lambda })  +|\beta^{(d)}(i\lambda)|^2 C_{\overline{\mu}}(e^{-i\lambda}) \rd) d\lambda
 \\
 \notag&&+\frac{1}{2\pi}\ip\frac{|\beta^{(d)}(i\lambda)|^2}{|\chi_{\overline{\mu}}^{(d)}(e^{-i\lambda})|^2}
 \ld(\chi_{\overline{\mu}}^{(d)}(e^{i\lambda})(A(e^{-i\lambda }))^{\top}g(\lambda)-(C_{\overline{\mu}}(e^{i\lambda}))^{\top}\rd)
p^{-1
}(\lambda) f(\lambda)p^{-1
}(\lambda)
\\
\notag &\quad&\times \ld(\chi_{\overline{\mu}}^{(d)}(e^{-i\lambda})g(\lambda)A(e^{i\lambda })-C_{\overline{\mu}}(e^{-i\lambda})\rd)d\lambda
 \\&=&\ld\langle \me S_{\overline{\mu}}
 \me a_{\overline{\mu}},\me P_{\overline{\mu}}^{-1}\me S_{\overline{\mu}}\me a_{\overline{\mu}}\rd\rangle+\ld\langle\me Q\me a,\me
 a\rd\rangle.\label{poh A_f_st.n_d}\end{eqnarray}
\end{thm}
\begin{proof}
See Appendix.
\end{proof}

 The filtering problem  for the functional  $A_N\vec\xi$ is solved directly by Theorem 3.1  by putting $\vec a(k)=\vec 0$ for
$k>N$. To solve the filtering problem for the $p$th coordinate of the single vector $\vec\xi (-N)$, we put $\vec a(N)=\vec\delta_p$, $\vec a(k)=\vec 0$ for $k\neq N$.

Define  a matrix $\me S_{\overline{\mu},N}$ with the $T\times T$ matrix  entries
 $(\me S_{\overline{\mu},N})_{l, m} =S^{\overline{\mu}}_{l+1,m-n(\gamma)}$ for $l\geq0$, $0\leq m\leq N+n(\gamma)$ and $(\me S_{\overline{\mu},N})_{l, m} =0$ for $l\geq0$, $m>N+n(\gamma)$.
Define another  matrix $\me Q_{N}$  with the $T\times T$ matrix entries $(\me Q_{N})_{l,k}=Q_{l,k}$, $0\leq l,k\leq N$, and $(\me Q_{N})_{l,k}=0$ otherwise.

The following corollaries take place.

\begin{nas}\label{nas_A_N_f_st.n_d}
A solution $\widehat{A}\vec\xi_N$ to the filtering problem for for a linear functional $A_N\xi$ of the values of a vector-valued stochastic sequence $\vec{\xi}(m)$  with
stationary  GM increments under condition  (12)  is  calculated by the formula
\be \label{otsinka A_N_f_st.n_d}
 \widehat{A}_N\vec\xi=\sum_{k=0}^N(\vec a(k))^{\top}(\vec\xi(-k)+\vec\eta(-k))-\ip
 (h_{\overline{\mu},N}(\lambda))^{\top}dZ_{\xi^{(n)}+\eta^{(n)}}(\lambda).
\ee
The spectral characteristic
$\vec h_{\overline{\mu},N}(\lambda)$ of the optimal estimate $\widehat{A}_N\vec\xi$ is calculated by the formula
\be \label{spectr A_N_f_st.n_d}
 (\vec h_{\overline{\mu}}(\lambda))^{\top}=\ld[\chi_{\overline{\mu}}^{(d)}(e^{i\lambda})(A_N(e^{-i\lambda }))^{\top}
  g(\lambda)-(C^{\overline{\mu}}_N(e^{i\lambda}))^{\top}\rd]p^{-1}(\lambda)\frac{\overline{\beta^{(d)}(i\lambda)}
}{\chi_{\overline{\mu}}^{(d)}(e^{i\lambda})},\ee
where
\[
 A_N(e^{i\lambda })=\sum_{k=0}^{N}\vec a(k)e^{-i\lambda k},\quad
 C_N^{\overline{\mu}}(e^{i\lambda})=\sum_{k=0}^{\infty}
 \ld(\me P_{\overline{\mu}}^{-1}\me S_{\overline{\mu},N}\me a_{\overline{\mu},N}\rd)_k e^{i\lambda(k+1)},\]
\[
     \vec {\me a}_{\overline{\mu},N}=(\vec  a_{\overline{\mu},N}(0),\vec  a_{\overline{\mu},N}(1),\ldots, \vec  a_{\overline{\mu},N}(N+n(\gamma)),0,\ldots )^{\top},\]
\[
   \vec   a_{\overline{\mu},N}(k)=\vec a_{-\overline{\mu},N}(k- n(\gamma)),\quad 0\leq k\leq N+n(\gamma),\]
\be\label{coeff a_N_mu_f_st.n_d}
 \vec a_{-\overline{\mu},N}(m)=\sum_{l=\max\ld\{m,0\rd\}}^{\min\{m+n(\gamma),N\}}e_{\gamma}(l-m)\vec a(l),\quad  -n(\gamma)\leq m\leq N.\ee
The  the value of the mean square error $\Delta(f,g;\widehat{A}_N\vec\xi)$ is calculated by the formula
\begin{eqnarray}
 \notag \Delta\ld(f,g;\widehat{A}_N\vec\xi\rd)&=&\Delta\ld(f,g;\widehat{A}\vec\eta_N\rd)= \mt E\ld|A\vec\eta_N-\widehat{A}\vec\eta_N\rd|^2
  \\&=&\ld\langle \me S_{\overline{\mu},N} \me
 a_{\overline{\mu},N},\me P_{\overline{\mu}}^{-1}\me S_{\overline{\mu},N}\me a_{\overline{\mu},N}\rd\rangle+\ld\langle\me Q_{N}\me a_N,\me
 a_{N}\rd\rangle,\label{poh A_N_f_st.n_d}\end{eqnarray}
where $\me a_N=((\vec a(0))^{\top},(\vec a(1))^{\top},(\vec a(2))^{\top},\ldots,(\vec a(N))^{\top},0,\ldots)^{\top}$.
\end{nas}

\begin{nas} \label{nas xi_f_st.n_d}
The optimal linear estimate $\widehat{\xi}_p(-N)$ of an unobserved value
$ \xi_p(-N)$, $N\geq0$, of the stochastic vector sequence $\vec\xi(m)$ with GM stationary increments based on observations of the sequence $\vec\xi(m)+\vec\eta(m)$ at points $m=0,-1,-2,\ldots$, where the noise sequence $\vec\eta(m)$ is uncorrelated with $\vec\xi(m)$, is calculated by the formula
\be\label{otsinka A_p_f_st.n_d}
 \widehat{\xi}_p(-N)=(\xi_p(-N)+\eta_p(-N))-\ip
 (\vec h_{\overline{\mu},N,p})^{\top}(\lambda)dZ_{\xi^{(n)}+\eta^{(n)}}(\lambda).\ee
The spectral characteristic of the  estimate $\widehat{\xi}_p(-N)$ is calculated by the formula
\be\label{spectr A_p_f_st.n_d}
 (\vec h_{\overline{\mu},N,p}(\lambda))^{\top}=
 \ld[e^{-i\lambda N }\chi_{\overline{\mu}}^{(d)}(e^{i\lambda})
 (\vec\delta_p)^{\top}g(\lambda)
 -(C_{N,p}^{\overline{\mu}}(e^{i\lambda}))^{\top}\rd]p^{-1}(\lambda)\frac{\overline{\beta^{(d)}(i\lambda)}
}{\chi_{\overline{\mu}}^{(d)}(e^{i\lambda})},\ee
where
\[
 C_{N,p}^{\overline{\mu}}(e^{i\lambda})=\sum_{k=0}^{\infty}
 \ld(\me P_{\overline{\mu}}^{-1}\me S_{\mu,N} \vec {\me a}_{\overline{\mu},N,p}\rd)_k e^{i\lambda
 (k+1)},\]
 \[
     \vec {\me a}_{\overline{\mu},N,p}=(0,\ldots,0,(\vec  a_{\overline{\mu},N,p}(N))^{\top},(\vec  a_{\overline{\mu},N,p}(N+1))^{\top},\ldots, (\vec  a_{\overline{\mu},N,p}(N+n(\gamma)))^{\top},0,\ldots )^{\top},\]
\[
   \vec   a_{\overline{\mu},N,p}(k)=\vec a_{-\overline{\mu},N,p}(k- n(\gamma)),\quad N\leq k\leq N+n(\gamma),\]
\[
 \vec a_{-\overline{\mu},N,p}(m)=e_{\gamma}(N-m)\vec \delta_p,\quad  N-n(\gamma)\leq m\leq N.\]

The value of the mean square error of the estimate $\widehat{\xi}_p(-N)$ is calculated by the formula
\begin{eqnarray}
\notag \Delta\ld(f,g;\widehat{\xi}_p(-N)\rd)&=&\Delta\Big(f,g;\widehat{\eta}_p(-N)\Big)= \mt E\Big|\eta_p(-N)-\widehat{\eta}_p(-N)\Big|^2
 \\&=& \ld\langle \me S_{\overline{\mu},N} \me
 a_{\overline{\mu},N.p},\me P_{\overline{\mu}}^{-1}\me S_{\overline{\mu},N}\me a_{\overline{\mu},N,p}\rd\rangle+\ld\langle Q_{0,0}\vec\delta_p,\vec\delta_p\rd\rangle.
\label{poh A_p_f_st.n_d}
\end{eqnarray}
\end{nas}

\begin{zau}
The filtering problem in the presence of fractional integration can be solved by  Theorem 3.1 and Corollaries 3.1, 3.2 under the conditions of  Theorem 2.2 on the increment  orders $d_i$.
\end{zau}

\subsection{Filtering of periodically stationary GM increment}

Consider \textbf{the filtering problem} for the functionals
\begin{equation}
A{\xi}=\sum_{k=0}^{\infty}{a}^{(\xi)}(k)\xi(-k), \quad
A_{M}{\xi}=\sum_{k=0}^{N}{a}^{(\xi)}(k)\xi(-k)
\end{equation}
which depend on unobserved values of the stochastic sequence $\xi(m)$ with periodically stationary
GM increments. Estimates are based on observations of the sequence $\xi(m)+\eta(m)$ at points $m=0,-1,-2,\ldots$, where the periodically stationary noise sequence $ \eta(m)$ is uncorrelated with $ \xi(m)$.

The functional $A{\xi}$ can be represented in the form
\begin{eqnarray}
\nonumber
A{\xi}& = &\sum_{k=0}^{\infty}{a}^{(\xi)}(k)\xi(-k)=\sum_{m=0}^{\infty}\sum_{p=1}^{T}
{a}^{(\xi)}(mT+p-1)\xi(-mT-p+1)
\\\nonumber
& = & \sum_{m=0}^{\infty}\sum_{p=1}^{T}a_p(m)\xi_p(-m)=\sum_{m=0}^{\infty}(\vec{a}(m))^{\top}\vec{\xi}(-m)
=A\vec{\xi},
\end{eqnarray}
where
\be \label{zeta1}
\vec{\xi}(m)=({\xi}_1(m),{\xi}_2(m),\dots,{\xi}_T(m))^{\top},\,
 {\xi}_p(m)=\xi(mT+p-1);\,p=1,2,\dots,T;
\ee
\be \label{azeta}
 \vec{a}(m) =({a}_1(m),{a}_2(m),\dots,{a}_T(m))^{\top},\,
 {a}_p(m)=a^{(\xi)}(mT+p-1);\,p=1,2,\dots,T.
\ee

In the same way, the functional $A{\eta}$ is represented as
\[
A{\eta}=\sum_{k=0}^{\infty}{a}^{(\xi)}(k)\eta(-k)=\sum_{m=0}^{\infty}(\vec{a}(m))^{\top}\vec{\eta}(-m)
=A\vec{\eta},
\]
\be \label{zeta2}
\vec{\eta}(m)=({\eta}_1(m),{\eta}_2(m),\dots,{\eta}_T(m))^{\top},\,
 {\eta}_p(m)=\eta(mT+p-1);\,p=1,2,\dots,T.
\ee

Making use of the introduced notations and statements of Theorem 3.1  we can claim that the following theorem holds true.

\begin{thm}
\label{thm_est_Azeta}
Let a stochastic sequence ${\xi}(m)$ with periodically stationary GM increments and a stochastic periodically stationary sequence ${\eta}(m)$   generate by formulas (27) and (29)
  vector-valued stochastic sequences $\vec{\xi}(m) $ and $\vec{\eta}(m) $ with  the spectral densities matrices $f(\lambda)=\{f_{ij}(\lambda)\}_{i,j=1}^{T}$ and $g(\lambda)=\{g_{ij}(\lambda)\}_{i,j=1}^{T}$. A solution $\widehat{A} \xi$ to the filtering problem for the functional $A\xi=A\vec \xi$ under conditions  (12) and (13) is  calculated by formula
 (15) for the coefficients $\vec a(m)$, $m\geq0$, defined in (28).
The spectral characteristic
$ \vec h_{\overline{\mu}}(\lambda)=\{h_{p}(\lambda)\}_{p=1}^{T}$ and the value of the mean square error $\Delta(f;\widehat{A}\xi)$ of the   estimate $\widehat{A}\xi$ are calculated by formulas
(17) and (18) respectively.
\end{thm}

The functional $A_M{\xi}$ can be represented in the form
\begin{eqnarray}
\nonumber
A_M{\xi}& = &\sum_{k=0}^{M}{a}^{(\xi)}(k)\zeta(-k)=\sum_{m=0}^{N}\sum_{p=1}^{T}
{a}^{(\xi)}(mT+p-1)\xi(-mT-p+1)
\\\nonumber
& = &\sum_{m=0}^{N}\sum_{p=1}^{T}a_p(m)\xi_p(-m)=\sum_{m=0}^{N}(\vec{a}(m))^{\top}\vec{\xi}(-m)=A_N\vec{\xi},
\end{eqnarray}
where $N=[\frac{M}{T}]$, the sequence   $\vec{\xi}(m) $ is determined by formula (27),
\begin{eqnarray}
\nonumber
(\vec{a}(m))^{\top}& = &({a}_1(m),{a}_2(m),\dots,{a}_T(m))^{\top},
\\\nonumber
 {a}_p(m)& = &a^{\zeta}(mT+p-1);\,0\leq m\leq N; 1\leq p\leq T;\,mT+p-1\leq M;
\\  {a}_p(N)& = &0;\quad
M+1\leq NT+p-1\leq (N+1)T-1;1\leq p\leq T. \label{aNzeta}
\end{eqnarray}

An estimate of a single unobserved value
$\xi(-M)$, $M\geq0$ of a stochastic sequence ${\xi}(m)$ with periodically stationary GM increments
is obtained by making use of the notations
$\xi(-M)=\xi_p(-N)=(\vec \delta_p)^{\top}\vec\xi(N)$, $N=[\frac{M}{T}]$, $p=M+1-NT$. We can conclude that the following corollaries hold true.

\begin{nas}
\label{nas_est_A_Nzeta}
Let a stochastic sequence ${\xi}(m)$ with periodically stationary GM increments and a stochastic periodically stationary sequence ${\eta}(m)$   generate by formulas (27) and (29)
  vector-valued stochastic sequences $\vec{\xi}(m) $ and $\vec{\eta}(m) $. A solution $\widehat{A}_M\xi$ to the filtering problem for  the functional $A_M\xi=A_N\vec{\xi}$  under condition (13) is  calculated by formula
 (19) for the coefficients $\vec a(m)$, $0\leq m \leq N$, defined in (30).
The spectral characteristic and the value of the mean square error of the estimate $\widehat{A}_M\xi$
are calculated by formulas  (20) and (22) respectively.
\end{nas}

\begin{nas}\label{nas zeta_e_d}
Let a stochastic sequence ${\xi}(m)$ with periodically stationary GM increments and a stochastic periodically stationary sequence ${\eta}(m)$   generate by formulas (27) and (29)
  vector-valued stochastic sequences $\vec{\xi}(m) $ and $\vec{\eta}(m) $. A solution $\widehat{\xi}(-M)$ to the filtering problem for an unobserved value $\xi(-M)=\xi_p(-N)=(\vec \delta_p)^{\top}\vec\xi(-N)$, $N=[\frac{M}{T}]$, $p=M+1-NT$,  under condition (13)
 is calculated by formula (23).
The spectral characteristic and the value of the mean square error of the estimate $\widehat{\xi}(-M)$ are calculated by   formulas
 (24) and (25) respectively.
\end{nas}

\section{Minimax (robust) method of filtering}\label{minimax_filtering}

Solutions of the problem of  estimating the functionals ${A}\vec\xi$ and ${A}_N\vec\xi$ constructed from unobserved values of the stochastic sequence $\vec{\xi}(m)$ with  stationary  GM increments
$\chi_{\overline{\mu},\overline{s}}^{(d)}(\vec{\xi}(m))$ having the spectral density matrix $f(\lambda)$
based on its observations with stationary  noise
$\vec\xi(m)+\vec\eta(m)$ at points $m=0,-1,-2,\dots$ are  proposed in Theorem 3.1 and Corollary 3.1  in the case where the spectral density matrices
$f(\lambda)$ and $g(\lambda)$ of the target sequence  and the noise are exactly known.

In this section, we study the case where the  complete information about the spectral density matrices is not available while the set   of admissible spectral densities $\md D=\md D_f\times\md D_g$ is known.
The minimax approach of estimation of the functionals from  unobserved values of stochastic sequences is considered, which
consists in finding an estimate that minimizes
the maximal values of the mean square errors for all spectral densities
from a class $\md D$ simultaneously. This method will be applied for the concrete classes of spectral densities.

The proceed with the stated problem, we recall the following definitions [33].

\begin{ozn}
For a given class of spectral densities $\mathcal{D}=\md
D_f\times\md D_g$, the spectral densities
$f^0(\lambda)\in\mathcal{D}_f$, $g^0(\lambda)\in\md D_g$
are called the least favourable densities in the class $\mathcal{D}$ for
optimal linear filtering of the functional $A\xi$ if the following relation holds true
\[
 \Delta(f^0,g^0)=\Delta(h(f^0,g^0);f^0,g^0)=
 \max_{(f,g)\in\mathcal{D}_f\times\md
 D_g}\Delta(h(f,g);f,g).\]
\end{ozn}

\begin{ozn}
For a given class of spectral
densities $\mathcal{D}=\md D_f\times\md D_g$ the spectral
characteristic $h^0(\lambda)$ of the optimal estimate of the functional
$A\xi$ is called minimax (robust) if the following relations hold true
\[
 h^0(\lambda)\in H_{\mathcal{D}}
 =\bigcap_{(f,g)\in\mathcal{D}_f\times\md D_g}L_2^{0}(p),\]
\[
 \min_{h\in H_{\mathcal{D}}}\max_{(f,g)\in \mathcal{D}_f\times\md D_g}\Delta(h;f,g)
 =\max_{(f,g)\in\mathcal{D}_f\times\md D_g}\Delta(h^0;f,g).\]
\end{ozn}

Taking into account the introduced definitions and the relations derived in the previous sections  we can verify that the following lemma holds true.

\begin{lema}
The spectral densities $f^0\in\mathcal{D}_f$,
$g^0\in\mathcal{D}_g$ which satisfy the minimality condition (13) are the least favourable spectral densities in the class $\mathcal{D}$ for the optimal
linear filtering of the functional $A\xi$  if the matrices $\me P_{\overline{\mu}}^0$, $\me S_{\overline{\mu}}^0$, $\me Q^0$ defined by
the Fourier coefficients of the functions
\[
 |\beta^{(d)}(i\lambda)|^2|\chi_{\overline{\mu}}^{(d)}(e^{-i\lambda})|^{-2}(p^0(\lambda))^{-1}, \,
 |\beta^{(d)}(i\lambda)|^2|\chi_{\overline{\mu}}^{(d)}(e^{-i\lambda})|^{-2}g^0(\lambda)(p^0(\lambda))^{-1},\,
 f^0(\lambda)g^0(\lambda)p^0(\lambda),\]
where
\[p^0(\lambda):=f^0(\lambda)+|\beta^{(d)}(i\lambda)|^2g^0(\lambda),\]
determine a solution of the constrained optimization problem
\be
 \max_{(f,g)\in \mathcal{D}_f\times\md D_g}\ld(\ld\langle \me S_{\overline{\mu}}
 \me a_{\overline{\mu}},\me P_{\overline{\mu}}^{-1}\me S_{\overline{\mu}}\me a_{\overline{\mu}}\rd\rangle+\ld\langle\me Q\me
 a,\me a\rd\rangle\rd)= \ld\langle \me S^0_{\overline{\mu}}
 \me a_{\overline{\mu}},(\me P^0_{\overline{\mu}})^{-1}\me S^0_{\overline{\mu}}\me a_{\overline{\mu}}\rd\rangle+\ld\langle\me Q^0\me
 a,\me a\rd\rangle.
\label{minimax1_e_d}
\ee
The minimax spectral characteristic $h^0=h_{\overline{\mu}}(f^0,g^0)$ is calculated by formula (17) if
$h_{\overline{\mu}}(f^0,g^0)\in H_{\mathcal{D}}$.
\end{lema}

The more detailed analysis of properties of the least favorable spectral densities and the minimax-robust spectral characteristics shows that the minimax spectral characteristic $h^0$ and the least favourable spectral densities $f^0$ and $g^0$ form a saddle
point of the function $\Delta(h;f,g)$ on the set
$H_{\mathcal{D}}\times\mathcal{D}$.
The saddle point inequalities
\[
 \Delta(h;f^0,g^0)\geq\Delta(h^0;f^0,g^0)\geq\Delta(h^0;f,g)\quad\forall (f,g)\in
 \mathcal{D},\forall h\in H_{\mathcal{D}}\]
hold true if $h^0=h_{\overline{\mu}}(f^0,g^0)$,
$h_{\overline{\mu}}(f^0,g^0)\in H_{\mathcal{D}}$ and $(f^0,g^0)$ is a solution of the constrained optimization problem
\be
 \widetilde{\Delta}(f)=-\Delta(h_{\overline{\mu}}(f^0,g^0);f,g)\to
 \inf,\quad (f,g)\in \mathcal{D},\label{zad_um_extr_e_d}
 \ee
where the functional $\Delta(h_{\overline{\mu}}(f^0,g^0);f,g)$ is calculated by the formula

\begin{eqnarray}
 \notag && \Delta(h_{\overline{\mu}}(f^0,g^0);f,g)
 \\
 \notag&=&\frac{1}{2\pi}\ip\frac{|\beta^{(d)}(i\lambda)|^2}{|\chi_{\overline{\mu}}^{(d)}(e^{-i\lambda})|^2}
 \ld(\chi_{\overline{\mu}}^{(d)}(e^{i\lambda})(A(e^{-i\lambda }))^{\top}g^0(\lambda)-(C^0_{\overline{\mu}}(e^{i\lambda}))^{\top}\rd)
\\
\notag &\quad&\times (p^0(\lambda))^{-1} f(\lambda)(p^0(\lambda))^{-1}\ld(\chi_{\overline{\mu}}^{(d)}(e^{-i\lambda})g^0(\lambda)A(e^{i\lambda })-C^0_{\overline{\mu}}(e^{-i\lambda})\rd)d\lambda
 \\
 \notag&&+\frac{1}{2\pi}\ip
 \frac{1}{|\chi_{\overline{\mu}}^{(d)}(e^{-i\lambda})|^2}
\ld(\chi_{\overline{\mu}}^{(d)}(e^{i\lambda})(A(e^{-i\lambda }))^{\top} f^0(\lambda)+|\beta^{(d)}(i\lambda)|^2(C^0_{\overline{\mu}}(e^{i\lambda}))^{\top}\rd)
\\
\notag &\quad &\times (p^0(\lambda))^{-1} g(\lambda)(p^0(\lambda))^{-1}\ld(\chi_{\overline{\mu}}^{(d)}(e^{-i\lambda})f^0(\lambda) A(e^{i\lambda })  +|\beta^{(d)}(i\lambda)|^2 C^0_{\overline{\mu}}(e^{-i\lambda}) \rd) d\lambda
 \\
 &=&\ld\langle \me S^0_{\overline{\mu}}
 \me a_{\overline{\mu}}, \ld(\me P^0_{\overline{\mu}}\rd)^{-1}\me S^0_{\overline{\mu}}\me a_{\overline{\mu}}\rd\rangle+\ld\langle\me Q^0\me a,\me
 a\rd\rangle\end{eqnarray}
with
\[
 C^0_{\overline{\mu}}(e^{i\lambda})=\sum_{k=0}^{\infty} \ld(\ld(\me
 P^0_{\overline{\mu}}\rd)^{-1}\me S^0_{\overline{\mu}}\me a_{\overline{\mu}}\rd)_ke^{i\lambda
 (k+1)}.\]

The constrained optimization problem (32) is equivalent to the unconstrained optimization problem
\be \label{zad_unconst_extr_f_st_d}
 \Delta_{\mathcal{D}}(f,g)=\widetilde{\Delta}(f,g)+ \delta(f,g|\mathcal{D})\to\inf,\ee
where $\delta(f,g|\mathcal{D})$ is the indicator function of the set
$\mathcal{D}$, namely $\delta(f|\mathcal{D})=0$ if $f\in \mathcal{D}$ and $\delta(f|\mathcal{D})=+\infty$ if $(f,g)\notin \mathcal{D}$.
The condition
 $0\in\partial\Delta_{\mathcal{D}}(f^0,g^0)$ characterizes a solution $(f^0,g^0)$ of the stated unconstrained optimization problem. This condition is the necessary and sufficient condition that the point $(f^0,g^0)$ belongs to the set of minimums of the convex functional $\Delta_{\mathcal{D}}(f,g)$
[10, 33, 34, 44].
Thus, we it allows us to find the equalities for the least favourable spectral densities in some special classes of spectral densities $\md D$.

The form of the functional $\Delta(h_{\overline{\mu}}(f^0,g^0);f,g)$ is suitable for application the Lagrange method of indefinite
multipliers to the constrained optimization problem (32).
Thus, the complexity of the problem is reduced to  finding the subdifferential of the indicator function of the set of admissible spectral densities. We illustrate the solving of the problem (34) for concrete sets admissible spectral densities  in the following subsections. A semi-uncertain filtering problem, when the spectral density $f(\lambda)$ is known and the spectral density $g(\lambda)$ belongs to  in class  $\md D_g$, is considered as well.

\subsection{Least favorable spectral density in classes $\md D_0\times\md D_{V}^{U}$}
\label{class_D0}

Consider the minimax filtering problem for the functional $A\vec{\xi}$
  for sets of admissible spectral densities $\md D_0^k$, $k=1,2,3,4$ of the sequence with GM increments $\vec \xi(m)$
$$\md D_{0}^{1} =\bigg\{f(\lambda )\left|\frac{1}{2\pi} \int
_{-\pi}^{\pi}
\frac{|\chi_{\overline{\mu}}^{(d)}(e^{-i\lambda})|^2}{|\beta^{(d)}(i\lambda)|^2}
f(\lambda )d\lambda  =P\right.\bigg\},$$
$$\md D_{0}^{2} =\bigg\{f(\lambda )\left|\frac{1}{2\pi }
\int _{-\pi }^{\pi}
\frac{|\chi_{\overline{\mu}}^{(d)}(e^{-i\lambda})|^2}{|\beta^{(d)}(i\lambda)|^2}
{\rm{Tr}}\,[ f(\lambda )]d\lambda =p\right.\bigg\},$$
$$\md D_{0}^{3} =\bigg\{f(\lambda )\left|\frac{1}{2\pi }
\int _{-\pi}^{\pi}
\frac{|\chi_{\overline{\mu}}^{(d)}(e^{-i\lambda})|^2}{|\beta^{(d)}(i\lambda)|^2}
f_{kk} (\lambda )d\lambda =p_{k}, k=\overline{1,T}\right.\bigg\},$$
$$\md D_{0}^{4} =\bigg\{f(\lambda )\left|\frac{1}{2\pi} \int _{-\pi}^{\pi}
\frac{|\chi_{\overline{\mu}}^{(d)}(e^{-i\lambda})|^2}{|\beta^{(d)}(i\lambda)|^2}
\left\langle B_{1} ,f(\lambda )\right\rangle d\lambda  =p\right.\bigg\},$$

\noindent
where  $p, p_k, k=\overline{1,T}$ are given numbers, $P, B_1$ are given positive-definite Hermitian matrices, and sets of admissible spectral densities $\md D_{V}^{U}$, $k=1,2,3,4$ for the stationary noise sequence $\vec \eta(m)$
\begin{equation*}
 {\md D_{V}^{U}} ^{1}=\left\{g(\lambda )\bigg|V(\lambda )\le g(\lambda
)\le U(\lambda ), \frac{1}{2\pi } \int _{-\pi}^{\pi}
g(\lambda )d\lambda=Q\right\},
\end{equation*}
\begin{equation*}
  {\md D_{V}^{U}} ^{2}  =\bigg\{g(\lambda )\bigg|{\mathrm{Tr}}\, [V(\lambda
)]\le {\mathrm{Tr}}\,[ g(\lambda )]\le {\mathrm{Tr}}\, [U(\lambda )],
\frac{1}{2\pi } \int _{-\pi}^{\pi}
{\mathrm{Tr}}\,  [g(\lambda)]d\lambda  =q \bigg\},
\end{equation*}
\begin{equation*}
{\md D_{V}^{U}} ^{3}  =\bigg\{g(\lambda )\bigg|v_{kk} (\lambda )  \le
g_{kk} (\lambda )\le u_{kk} (\lambda ),
\frac{1}{2\pi} \int _{-\pi}^{\pi}
g_{kk} (\lambda
)d\lambda  =q_{k} , k=\overline{1,T}\bigg\},
\end{equation*}
\begin{equation*}
{\md D_{V}^{U}} ^{4}  =\bigg\{g(\lambda )\bigg|\left\langle B_{2}
,V(\lambda )\right\rangle \le \left\langle B_{2},g(\lambda
)\right\rangle \le \left\langle B_{2} ,U(\lambda)\right\rangle,
\frac{1}{2\pi }
\int _{-\pi}^{\pi}
\left\langle B_{2},g(\lambda)\right\rangle d\lambda  =q\bigg\},
\end{equation*}
where the spectral densities $V( \lambda ),U( \lambda )$ are known and fixed, $ q$,  $q_k$, $k=\overline{1,T}$ are given numbers, $Q$, $B_2$ are given positive definite Hermitian matrices.

From the condition $0\in\partial\Delta_{\mathcal{D}}(f^0,g^0)$
we find the following equations which determine the least favourable spectral densities for these given sets of admissible spectral densities.

Define
\[C^{f0}_{\overline{\mu}}(e^{i\lambda}):=
\chi_{\overline{\mu}}^{(d)}(e^{i\lambda})(A(e^{-i\lambda }))^{\top}g^0(\lambda)-(C^0_{\overline{\mu}}(e^{i\lambda}))^{\top},
\]
\[C^{g0}_{\overline{\mu}}(e^{i\lambda}):=\frac{|\chi_{\overline{\mu}}^{(d)}(e^{-i\lambda})|^2}{|\beta^{(d)}(i\lambda)|^2}(A(e^{-i\lambda }))^{\top} f^0(\lambda)+\chi_{\overline{\mu}}^{(d)}(e^{-i\lambda})(C^0_{\overline{\mu}}(e^{i\lambda}))^{\top}.
\]
For the first pair of the sets of admissible spectral densities $\md D_{f0}^1\times {\md D_{V}^{U}} ^{1}$ we have equations
\begin{equation} \label{eq_4_1}
C^{f0}_{\overline{\mu}}(e^{i\lambda})
C^{f0}_{\overline{\mu}}(e^{i\lambda})^{*}
=\left(\frac{|\chi_{\overline{\mu}}^{(d)}(e^{-i\lambda})|^2}{|\beta^{(d)}(i\lambda)|^2} p^0(\lambda )\right)
\vec{\alpha}\cdot \vec{\alpha}^{*}\left(\frac{|\chi_{\overline{\mu}}^{(d)}(e^{-i\lambda})|^2}{|\beta^{(d)}(i\lambda)|^2} p^0(\lambda )\right),
\end{equation}
\begin{multline}
\label{eq_5_1}
 C^{g0}_{\overline{\mu}}(e^{i\lambda})C^{g0}_{\overline{\mu}}(e^{i\lambda})^{*}=\\
=
\left(\frac{|\chi_{\overline{\mu}}^{(d)}(e^{-i\lambda})|^2}{|\beta^{(d)}(i\lambda)|^2} p^0 (\lambda )\right)(\vec{\beta}\cdot \vec{\beta}^{*}+\Gamma _{1} (\lambda )+\Gamma _{2} (\lambda )) \left(\frac{|\chi_{\overline{\mu}}^{(d)}(e^{-i\lambda})|^2}{|\beta^{(d)}(i\lambda)|^2} p^0 (\lambda )\right),
\end{multline}
where $\vec{\alpha}$, $ \vec{\beta}$ are vectors of Lagrange multipliers, $\Gamma _{1} (\lambda )\le 0$ and $\Gamma _{1} (\lambda )=0$ if $g^0(\lambda )>V(\lambda )$, $
\Gamma _{2} (\lambda )\ge 0$ and $\Gamma _{2} (\lambda )=0$ if $g^0(\lambda )<U(\lambda )$.

For the second pair of the sets of admissible spectral densities $\md D_{f0}^2\times {\md D_{V}^{U}} ^{2}$ we have equation
\begin{equation}  \label{eq_4_2}
C^{f0}_{\overline{\mu}}(e^{i\lambda})C^{f0}_{\overline{\mu}}(e^{i\lambda})^{*}=
\alpha^{2} \left(\frac{|\chi_{\overline{\mu}}^{(d)}(e^{-i\lambda})|^2}{|\beta^{(d)}(i\lambda)|^2} p^0 (\lambda )\right)^{2},
\end{equation}
\begin{equation}   \label{eq_5_2}
C^{g0}_{\overline{\mu}}(e^{i\lambda})C^{g0}_{\overline{\mu}}(e^{i\lambda})^{*}
=
(\beta^{2} +\gamma _{1} (\lambda )+\gamma _{2} (\lambda ))\left(\frac{|\chi_{\overline{\mu}}^{(d)}(e^{-i\lambda})|^2}{|\beta^{(d)}(i\lambda)|^2} p^0 (\lambda )\right)^2,
\end{equation}
where $\alpha^{2}$, $ \beta^{2}$ are Lagrange multipliers,  $\gamma _{1} (\lambda )\le 0$ and $\gamma _{1} (\lambda )=0$ if ${\mathrm{Tr}}\,
[g^0(\lambda )]> {\mathrm{Tr}}\,  [V(\lambda )]$, $\gamma _{2} (\lambda )\ge 0$ and $\gamma _{2} (\lambda )=0$ if $ {\mathrm{Tr}}\,[g^0(\lambda )]< {\mathrm{Tr}}\, [ U(\lambda)]$.

For the third pair of the sets of admissible spectral densities $\md D_{f0}^3\times {\md D_{V}^{U}}^{3}$ we have equation

\begin{equation}   \label{eq_4_3}
C^{f0}_{\overline{\mu}}(e^{i\lambda})C^{f0}_{\overline{\mu}}(e^{i\lambda})^{*}
=
\left(\frac{|\chi_{\overline{\mu}}^{(d)}(e^{-i\lambda})|^2}{|\beta^{(d)}(i\lambda)|^2} p^0(\lambda )\right)\left\{\alpha _{k}^{2} \delta _{kl} \right\}_{k,l=1}^{T} \left(\frac{|\chi_{\overline{\mu}}^{(d)}(e^{-i\lambda})|^2}{|\beta^{(d)}(i\lambda)|^2} p^0(\lambda )\right),
\end{equation}
\begin{multline}   \label{eq_5_3}
C^{g0}_{\overline{\mu}}(e^{i\lambda})C^{g0}_{\overline{\mu}}(e^{i\lambda})^{*}=\\
=
\left(\frac{|\chi_{\overline{\mu}}^{(d)}(e^{-i\lambda})|^2}{|\beta^{(d)}(i\lambda)|^2} p^0 (\lambda )\right) \left\{(\beta_{k}^{2} +\gamma _{1k} (\lambda )+\gamma _{2k} (\lambda ))\delta _{kl}\right\}_{k,l=1}^{T} \left(\frac{|\chi_{\overline{\mu}}^{(d)}(e^{-i\lambda})|^2}{|\beta^{(d)}(i\lambda)|^2} p^0 (\lambda )\right),
\end{multline}
where  $\alpha _{k}^{2}$, $ \beta_{k}^{2}$ are Lagrange multipliers, $\delta _{kl}$ are Kronecker symbols, $\gamma _{1k} (\lambda )\le 0$ and $\gamma _{1k} (\lambda )=0$ if $g_{kk}^{0} (\lambda )>v_{kk} (\lambda )$, $\gamma _{2k} (\lambda )\ge 0$ and $\gamma _{2k} (\lambda )=0$ if $g_{kk}^{0} (\lambda )<u_{kk} (\lambda)$.

For the fourth pair of the sets of admissible spectral densities $\md D_{f0}^4\times {\md D_{V}^{U}}^{4}$ we have equation
\begin{equation}   \label{eq_4_4}
C^{f0}_{\overline{\mu}}(e^{i\lambda})C^{f0}_{\overline{\mu}}(e^{i\lambda})^{*}
=
\alpha^{2} \left(\frac{|\chi_{\overline{\mu}}^{(d)}(e^{-i\lambda})|^2}{|\beta^{(d)}(i\lambda)|^2} p^0 (\lambda )\right)B_{1}^{\top} \left(\frac{|\chi_{\overline{\mu}}^{(d)}(e^{-i\lambda})|^2}{|\beta^{(d)}(i\lambda)|^2} p^0(\lambda )\right),
\end{equation}
\begin{multline}    \label{eq_5_4}
C^{g0}_{\overline{\mu}}(e^{i\lambda})C^{g0}_{\overline{\mu}}(e^{i\lambda})^{*}=\\
=
(\beta^{2} +\gamma'_{1}(\lambda )+\gamma'_{2}(\lambda ))\left(\frac{|\chi_{\overline{\mu}}^{(d)}(e^{-i\lambda})|^2}{|\beta^{(d)}(i\lambda)|^2} p^0 (\lambda )\right) B_2^\top  \left(\frac{|\chi_{\overline{\mu}}^{(d)}(e^{-i\lambda})|^2}{|\beta^{(d)}(i\lambda)|^2} p^0 (\lambda )\right),
\end{multline}
where $\alpha^{2}$, $\beta^{2}$ are Lagrange multipliers, $\gamma'_{1}( \lambda )\le 0$ and $\gamma'_{1} ( \lambda )=0$ if $\langle B_{2},g^0( \lambda) \rangle > \langle B_{2},V( \lambda ) \rangle$, $\gamma'_{2}( \lambda )\ge 0$ and $\gamma'_{2} ( \lambda )=0$ if $\langle
B_{2} ,g^0( \lambda) \rangle < \langle B_{2} ,U( \lambda ) \rangle$.

The following theorem  holds true.

\begin{thm}
Let the minimality condition (13) hold true. The least favorable spectral densities $f^0(\lambda)$ and $g^0(\lambda)$ in the classes $ \md  D_0^{k}\times{\md D_{V}^{U}}^{k}$, $k=1,2,3,4$, for the optimal linear filtering  of the functional  $A\vec{\xi}$   are determined by  pairs of equations (35)-(36), (37)-(38), (39)-(40), (41)-(42),
the constrained optimization problem (31) and restrictions  on densities from the corresponding classes $ \md  D_0^{k}\times{\md D_{V}^{U}}^{k}$, $k=1,2,3,4$.  The minimax-robust spectral characteristic $h_{\overline{\mu}}(f^0,g^0)$ of the optimal estimate of the functional $A\vec{\xi}$ is determined by the formula (17).
\end{thm}

\subsection{Semi-uncertain filtering problem in classes $\md D_{\eps}$ of least favorable noise spectral density}

Consider a semi-uncertain filtering problem for the functional $A\vec{\xi}$, where the spectral density $f(\lambda)$ of the  sequence with GM increments $\vec \xi(m)$ is known and   the spectral density $g(\lambda)$ of the stationary noise sequence  $\vec \eta(m)$ belongs to the
   sets of admissible spectral densities $\md D_{\eps}^k,k=1,2,3,4$
   \begin{equation*}
\md D_{\varepsilon }^{1}  =\bigg\{g(\lambda )\bigg|{\mathrm{Tr}}\,
[g(\lambda )]=(1-\varepsilon ) {\mathrm{Tr}}\,  [g_{1} (\lambda
)]+\varepsilon {\mathrm{Tr}}\,  [W(\lambda )],
\frac{1}{2\pi} \int _{-\pi}^{\pi}
{\mathrm{Tr}}\,
[g(\lambda )]d\lambda =q \bigg\};
\end{equation*}
\begin{equation*}
\md D_{\varepsilon }^{2}  =\bigg\{g(\lambda )\bigg|g_{kk} (\lambda)
=(1-\varepsilon )g_{kk}^{1} (\lambda )+\varepsilon w_{kk}(\lambda),
\frac{1}{2\pi} \int _{-\pi}^{\pi}
g_{kk} (\lambda)d\lambda  =q_{k} , k=\overline{1,T}\bigg\};
\end{equation*}
\begin{equation*}
\md D_{\varepsilon }^{3} =\bigg\{g(\lambda )\bigg|\left\langle B_{1},g(\lambda )\right\rangle =(1-\varepsilon )\left\langle B_{2},g_{1} (\lambda )\right\rangle+\varepsilon \left\langle B_{2},W(\lambda )\right\rangle,
\frac{1}{2\pi}\int _{-\pi}^{\pi}
\left\langle B_{2} ,g(\lambda )\right\rangle d\lambda =q\bigg\};
\end{equation*}
\begin{equation*}
\md D_{\varepsilon }^{4}=\bigg\{g(\lambda )\bigg|g(\lambda)=(1-\varepsilon )g_{1} (\lambda )+\varepsilon W(\lambda ),
\frac{1}{2\pi } \int _{-\pi}^{\pi}
g(\lambda )d\lambda=Q\bigg\},
\end{equation*}
where  $g_{1} ( \lambda )$ is a fixed spectral density, $W(\lambda)$ is an unknown spectral density, $q,  q_k, k=\overline{1,T}$, are given numbers, $Q, B_2$ are  given positive-definite Hermitian matrices.

The condition $0\in\partial\Delta_{\mathcal{D}}(f,g^0)$
implies the equations which determine the least favourable spectral densities of the noise sequence $\vec \eta(m)$.

Define
\[C^{g,semi}_{\overline{\mu}}(e^{i\lambda}):=\frac{|\chi_{\overline{\mu}}^{(d)}(e^{-i\lambda})|^2}{|\beta^{(d)}(i\lambda)|^2}(A(e^{-i\lambda }))^{\top} f(\lambda)+\chi_{\overline{\mu}}^{(d)}(e^{-i\lambda})(C^{semi}_{\overline{\mu}}(e^{i\lambda}))^{\top},
\]
where $C^{semi}_{\overline{\mu}}(e^{i\lambda})$ is determined by the matrices $\me P_{\overline{\mu}}^{semi}$, $\me S_{\overline{\mu}}^{semi}$  defined by
the Fourier coefficients of the functions
\[
 |\beta^{(d)}(i\lambda)|^2|\chi_{\overline{\mu}}^{(d)}(e^{-i\lambda})|^{-2}(f(\lambda)+|\beta^{(d)}(i\lambda)|^2g^0(\lambda))^{-1}, \]
\[
 |\beta^{(d)}(i\lambda)|^2|\chi_{\overline{\mu}}^{(d)}(e^{-i\lambda})|^{-2}g^0(\lambda)(f(\lambda)+|\beta^{(d)}(i\lambda)|^2g^0(\lambda))^{-1}.\]

For the  sets of admissible spectral densities $\md D_{\varepsilon }^{k}$, $k=1,2,3,4$, we have equations, respectively,
\begin{equation}\label{eq_6_1_semi}
  C^{g,semi}_{\overline{\mu}}(e^{i\lambda})C^{g,semi}_{\overline{\mu}}(e^{i\lambda})^{*}
=
(\beta^{2} +\gamma_2(\lambda ))\left(\frac{|\chi_{\overline{\mu}}^{(d)}(e^{-i\lambda})|^2}{|\beta^{(d)}(i\lambda)|^2} \ld(f(\lambda)+|\beta^{(d)}(i\lambda)|^2g^0(\lambda)\rd)\right)^{2},
\end{equation}
where   $\beta^{2}$ is a Lagrange multiplier, $\gamma_2(\lambda )\le 0$ and $\gamma_2(\lambda )=0$ if ${\mathrm{Tr}}\,[g^0(\lambda )]>(1-\varepsilon ) {\mathrm{Tr}}\, [g_{1} (\lambda )]$,
\begin{multline}   \label{eq_6_2_semi}
C^{g,semi}_{\overline{\mu}}(e^{i\lambda})C^{g,semi}_{\overline{\mu}}(e^{i\lambda})^{*}
=
 \frac{|\chi_{\overline{\mu}}^{(d)}(e^{-i\lambda})|^2}{|\beta^{(d)}(i\lambda)|^2} \ld(f(\lambda)+|\beta^{(d)}(i\lambda)|^2g^0(\lambda)\rd)
\\
\times\left\{(\beta_{k}^{2} +\gamma_{k}^2 (\lambda ))\delta _{kl} \right\}_{k,l=1}^{T}  \frac{|\chi_{\overline{\mu}}^{(d)}(e^{-i\lambda})|^2}{|\beta^{(d)}(i\lambda)|^2} \ld(f(\lambda)+|\beta^{(d)}(i\lambda)|^2g^0(\lambda)\rd) ,
\end{multline}
where   $\beta_{k}^{2}$ are Lagrange multipliers, $\gamma_{k}^2(\lambda )\le 0$ and $\gamma_{k}^2 (\lambda )=0$ if $g_{kk}^{0}(\lambda )>(1-\varepsilon )g_{kk}^{1} (\lambda )$,
\begin{multline}   \label{eq_6_3_semi}
C^{g,semi}_{\overline{\mu}}(e^{i\lambda})C^{g,semi}_{\overline{\mu}}(e^{i\lambda})^{*}
=
(\beta^{2} +\gamma_2'(\lambda )) \frac{|\chi_{\overline{\mu}}^{(d)}(e^{-i\lambda})|^2}{|\beta^{(d)}(i\lambda)|^2} \ld(f(\lambda)+|\beta^{(d)}(i\lambda)|^2g^0(\lambda)\rd)
\\
\times B_{2}^{\top}  \frac{|\chi_{\overline{\mu}}^{(d)}(e^{-i\lambda})|^2}{|\beta^{(d)}(i\lambda)|^2} \ld(f(\lambda)+|\beta^{(d)}(i\lambda)|^2g^0(\lambda)\rd),
\end{multline}
where  $\beta^{2}$ is a Lagrange multiplier, $\gamma_2' ( \lambda )\le 0$ and $\gamma_2' ( \lambda )=0$ if $\langle B_{2},g^0(\lambda) \rangle>(1-\varepsilon) \langle B_{2} ,g_{1}(\lambda) \rangle$,
\begin{multline}  \label{eq_6_4_semi}
C^{g,semi}_{\overline{\mu}}(e^{i\lambda})C^{g,semi}_{\overline{\mu}}(e^{i\lambda})^{*}
=
\frac{|\chi_{\overline{\mu}}^{(d)}(e^{-i\lambda})|^2}{|\beta^{(d)}(i\lambda)|^2} \ld(f(\lambda)+|\beta^{(d)}(i\lambda)|^2g^0(\lambda)\rd)
\\
\times(\vec{\beta}\cdot \vec{\beta}^{*}+\Gamma(\lambda)) \frac{|\chi_{\overline{\mu}}^{(d)}(e^{-i\lambda})|^2}{|\beta^{(d)}(i\lambda)|^2} \ld(f(\lambda)+|\beta^{(d)}(i\lambda)|^2g^0(\lambda)\rd),
\end{multline}
where   $\vec{\beta}$ is a vector of Lagrange multipliers, $\Gamma(\lambda )\le 0$ and $\Gamma(\lambda )=0$ if $g^0(\lambda )>(1-\varepsilon )g_{1} (\lambda )$.

The following theorem  holds true.

\begin{thm}
Let the minimality condition (13) hold true. The least favorable spectral density  $g^0(\lambda)$ in the classes $ \md D_{\varepsilon }^{k}$, $k=1,2,3,4$, for the optimal linear filtering  of the functional  $A\vec{\xi}$, when the spectral density $f(\lambda)$ is known,   is determined by   equations (43), (44), (45), (46),
the constrained optimization problem (31) for the fixed $f(\lambda)$ and the restrictions  on density from the corresponding classes $  \md D_{\varepsilon }^{k}$, $k=1,2,3,4$.  The minimax-robust spectral characteristic $h_{\overline{\mu}}(f,g^0)$ of the optimal estimate of the functional $A\vec{\xi}$ is determined by the formula (17).
\end{thm}

\section*{Appendix}

\emph{Proof of Theorem 3.1.}

The condition $(A\vec\eta-\widehat{A}\vec\eta)\perp
H^{0}(\xi^{(d)}+\eta^{(d)})$ implies   the relation which holds true for all $k\leq0$:
\begin{eqnarray*}
 &&\mt E\ld(A\vec\eta-\widehat{A}\vec\eta\rd)^{\top}
 \ld(\overline{\chi_{\overline{\mu},\overline{s}}^{(d)}(\vec{\xi}(m))
+\chi_{\overline{\mu},\overline{s}}^{(d)}(\vec{\eta}(m))}\rd)
 \\
 &=&\frac{1}{2\pi}\ip
 \ld(A(e^{-i\lambda })-\beta^{(d)}(i\lambda)\vec h_{\overline{\mu}}(\lambda)\rd)^{\top}g(\lambda) e^{-i\lambda
 k}\chi_{\overline{\mu}}^{(d)}(e^{i\lambda})d\lambda
 \\
 &&-\frac{1}{2\pi}\ip (\vec h_{\overline{\mu}}(\lambda))^{\top}f(\lambda)e^{-i\lambda
 k}\chi_{\overline{\mu}}^{(d)}(e^{i\lambda})\frac{1}{\overline{\beta^{(d)}(i\lambda)}}d\lambda=0.\end{eqnarray*}
Thus, for all $k\leq0$, the function $\vec h_{\overline{\mu}}(\lambda)$ satisfies the relation
\[
 \ip\ld[(A(e^{-i\lambda}))^{\top}g(\lambda)\overline{\beta^{(d)}(i\lambda)}-(\vec h_{\overline{\mu}}(\lambda))^{\top}
 p(\lambda)\rd]\frac{\chi_{\overline{\mu}}^{(d)}(e^{i\lambda})}{\overline{\beta^{(d)}(i\lambda)}}e^{-i\lambda
 k}d\lambda=\vec 0.\]
 The latter relation implies  that   the spectral characteristic
$\vec h_{\overline{\mu}}(\lambda)$ allows a representation of the form (17) with
\[
C_{\overline{\mu}}(e^{i\lambda})=\sum_{k=0}^{\infty}\vec c_{\overline{\mu}}(k)e^{i\lambda
(k+1)},
\]
where $\vec c_{\overline{\mu}}(k)=\{c_{\overline{\mu} p}(k)\}_{p=1}^T$, $k\geq0$, are unknown coefficients to be found.

The condition $ \widehat{A}\vec\eta\in H^{0}(\xi^{(d)}+\eta^{(d)}) $ implies the following representation of  the
spectral characteristic $\vec h_{\overline{\mu}}(\lambda)$:
\[\vec h_{\overline{\mu}}(\lambda)=h (\lambda)\chi_{\overline{\mu}}^{(d)}(e^{-i\lambda})
\frac{1}{\beta^{(d)}(i\lambda)}, \quad h (\lambda)=
\sum_{k=0}^{\infty}\vec s(k)e^{-i\lambda k},\]
which leads to  the relation holding true for all $l\geq 1$:
\be
 \ip \ld[ \chi_{\overline{\mu}}^{(d)}(e^{i\lambda})(A(e^{-i\lambda
 }))^{\top}g(\lambda)
 - (C_{\overline{\mu}}(e^{i\lambda}))^{\top}\rd]p^{-1}(\lambda)
\frac{|\beta^{(d)}(i\lambda)|^2}{|\chi_{\overline{\mu}}^{(d)}(e^{-i\lambda})|^2
 }e^{-i\lambda
 l}d\lambda=\vec 0.\label{spivv_um_1_f_st.n_d}\ee

In terms of the  matrix-valued Fourier coefficients $S^{\overline{\mu}}_{k,j}$ for $k\geq0,j\geq-n(\gamma)$, $P_{k,j}^{\overline{\mu}}$, $Q_{k,j}$ for $k,j\geq0$,  relation (47) is presented as a system of linear equations
\[
 \sum_{ m=-n(\gamma)}^{\infty}S^{\overline{\mu}}_{l+1,m}\vec a_{-\overline{\mu}}(m)
 =\sum_{k=0}^{\infty}P_{l+1,k+1}^{\overline{\mu}}\vec c_{\overline{\mu}}(k),\quad l\geq0,\]
or in the matrix form as
\[
 \me S_{\overline{\mu}}\me a_{\overline{\mu}}=\me P_{\overline{\mu}}\me c_{\overline{\mu}},\]
where the vectors $$\me c_{\overline{\mu}}=((\vec c_{\overline{\mu}}(0))^{\top},(\vec c_{\overline{\mu}}(1))^{\top},(\vec c_{\mu }(2))^{\top}, \ldots)^{\top}, \quad \me a_{\overline{\mu}}=((\vec a_{\overline{\mu}}(0))^{\top},
(\vec a_{\overline{\mu}}(1))^{\top},(\vec a_{\overline{\mu}}(2))^{\top},\ldots)^{\top},$$ the coefficients $ a_{\overline{\mu}}(k)=a_{-\mu}(k-n(\gamma))$, $k\geq0$ and the linear operators $\me P_{\overline{\mu}}$, $\me S_{\overline{\mu}}$ are    defined in  Subsection 3.1.
 Therefore, the unknown coefficients $\vec c_{\overline{\mu}}(k)$, $k\geq0$, are calculated by the formula
\[
 \vec c_{\overline{\mu}}(k)=\ld(\me P_{\overline{\mu}}^{-1}\me S_{\overline{\mu}}\me a_{\overline{\mu}}\rd)_k,\quad k\geq 0,\]
where $\ld(\me
P_{\overline{\mu}}^{-1}\me S_{\overline{\mu}}\me a_{\overline{\mu}}\rd)_k$, $k\geq 0$, is the
 $k$th $T$-dimensional vector element of the vector $\me P_{\overline{\mu}}^{-1}\me S_{\overline{\mu}}\me a_{\overline{\mu}}$.

The derived expressions justifies the formulas (17) and (18) for calculating the spectral characteristic $\vec h_{\overline{\mu}}(\lambda)$ and the value of the mean square error $\Delta\ld(f,g;\widehat{A}\vec\xi\rd)$ of the estimate $\widehat{A}\vec\xi$ of the functional $A\vec\xi$. $\square$

\end{document}